\documentclass[11pt]{article}

\usepackage{amsmath,amsthm}
\usepackage{amssymb}
\usepackage{hyperref}
\usepackage{color}

\allowdisplaybreaks

\begin{document}

\title{Euler sums of generalized alternating hyperharmonic numbers II}
\author{
Rusen Li
%\thanks{
%}
\\
%\small Department of Mathematical Sciences, School of Science\\
%\small Zhejiang Sci-Tech University\\
%\small Hangzhou 310018 China\\
\small School of Mathematics\\
\small Shandong University\\
\small Jinan 250100 China\\
%\small \texttt{komatsu@zstu.edu.cn}\\
\small \texttt{limanjiashe@163.com}
}

\date{
\small 2020 MR Subject Classifications: 11B37, 11B68, 11M06
%\small Submitted: November 15, 2019;  Accepted: December 25, 2019.\\
%\small MR Subject Classifications: Primary 11B65; Secondary 05A19
}

\maketitle

\def\stf#1#2{\left[#1\atop#2\right]}
\def\sts#1#2{\left\{#1\atop#2\right\}}
\def\e{\mathfrak e}
\def\f{\mathfrak f}

\newtheorem{theorem}{Theorem}
\newtheorem{Prop}{Proposition}
\newtheorem{Cor}{Corollary}
\newtheorem{Lem}{Lemma}
\newtheorem{Example}{Example}
\newtheorem{Remark}{Remark}
\newtheorem{Definition}{Definition}
\newtheorem{Conjecture}{Conjecture}
\newtheorem{Problem}{Problem}

\begin{abstract}
In this paper, we introduce a new type of generalized alternating hyperharmonic numbers $H_n^{(p,r,s_{1},s_{2})}$, and show that Euler sums of the generalized alternating hyperharmonic numbers $H_n^{(p,r,s_{1},s_{2})}$ can be expressed in terms of linear combinations of classical (alternating) Euler sums.
\end{abstract}

\section{Introduction and preliminaries}

The generalized alternating hyperharmonic numbers (of type I) $H_n^{(p,r,1)}$ are introduced by the author \cite{LiRusen} which are defined as
\begin{align*}
H_n^{(p,r,1)}:=\sum_{k=1}^n (-1)^{k-1} H_{k}^{(p,r-1,1)}\quad (n,p,r \in \mathbb N :=\{1,2,3,\cdots\}, H_n^{(p,1,1)}=H_n^{(p)}=\sum_{j=1}^{n}\frac{1}{j^{p}})\,.
\end{align*}
These numbers are alternating analogues of the classical generalized hyperharmonic numbers which are defined by (see {\cite{Dil,omur}})
$$
H_n^{(p,r)}:=\sum_{j=1}^n H_{j}^{(p,r-1)} \quad (n,p,r \in \mathbb N),
$$
where $H_n^{(p,1)}=H_n^{(p)}$ are the classical generalized harmonic numbers. It is obvious that $H_n^{(1,r)}=h_n^{(r)}$ are the classical hyperharmonic numbers {\cite{Benjamin,Conway}}. In particular $H_n^{(1,1)}=H_n$ are the classical harmonic numbers. It is obvious that these harmonic numbers play an essential role in number theory and combinatorics. For some applications in analysis of algorithms and some other areas, please see Knuth's book \cite{Knuth}.

In this paper, as a natural consideration, we introduce a new type of generalized alternating hyperharmonic numbers $H_n^{(p,r,s_{1},s_{2})}$, which are unified extension of the generalized hyperharmonic numbers $H_n^{(p,r)}$ and the generalized alternating hyperharmonic numbers (of type I) $H_n^{(p,r,1)}$. We also study the properties of their Euler sums of type $\sum_{n=1}^\infty H_n^{(p,r,s_{1},s_{2})}/{{n}^{q}}$. We will show that their Euler sums can be expressed in terms of linear combinations of classical (alternating) Euler sums.

In the spirit of Flajolet and Salvy \cite{Flajolet}, we write four types of linear (alternating) Euler sums as
\begin{align*}
S_{p,q}^{+,+}:=\sum_{n=1}^\infty \frac{H_n^{(p)}}{{n}^{q}},
\quad S_{p,q}^{+,-}:=\sum_{n=1}^\infty (-1)^{n-1}\frac{H_n^{(p)}}{{n}^{q}},\\
S_{p,q}^{-,+}:=\sum_{n=1}^\infty \frac{\overline{H}_n^{(p)}}{{n}^{q}},
\quad S_{p,q}^{-,-}:=\sum_{n=1}^\infty (-1)^{n-1}\frac{\overline{H}_n^{(p)}}{{n}^{q}}\,,
\end{align*}
where $\overline{H}_n^{(p)}=\sum_{j=1}^n (-1)^{j-1}/{j^p}$ denotes the classical generalized alternating harmonic numbers. When $p \geq 0$, $H_n^{(-p)}$ and $\overline{H}_n^{(-p)}$ are understood to be the sum
$\sum_{j=1}^{n} j^p$ and $\sum_{j=1}^{n} (-1)^{j-1} j^p$, respectively.

For convenience, we recall the well-known Hurwitz zeta function defined by
$$
\zeta(s,a)=\sum_{n=0}^\infty \frac{1}{(n+a)^{s}} \quad (s\in \mathbb C, \mathfrak Re(s)>1, a>0).
$$
Note that $\mathbb C$ denotes the set of complex numbers and $\mathfrak Re(s)$ denotes the real part of the complex number $s$. When $a=1$, $\zeta(s,1)$ is the famous Riemann zeta function $\zeta(s):=\sum_{n=1}^\infty n^{-s}$.

Euler sums of harmonic numbers have been studied by many famous mathematicians since the time of Euler, for the reason that the classical (alternating) Euler sums are closely related to the Riemann zeta function. A well-known result \cite{Flajolet} that can be traced back to the time of Euler is as the following:
$$
2\sum_{n=1}^\infty \frac{H_n}{n^{m}}=(m+2)\zeta(m+1)-\sum_{n=1}^{m-2}\zeta(m-n)\zeta(n+1), \quad m=2,3,\cdots.
$$

By using numerical computations, Bailey, Borwein and Girgensohn \cite{Bailey} determined,  whether or not a particular infinite sum involving the generalized harmonic numbers $H_n^{(m)}$ could be expressed as a rational linear combination of several given constants.

Since hyperharmonic and generalized (alternating) hyperharmonic numbers are hyper-generalizations of harmonic numbers, their Euler sums  (see \cite{Dil,Flajolet,Kamano,Rusen,LiRusen,Matsuoka,mezo} for example) may have interesting relations with the Riemann zeta function.

The most powerful method for computing Euler sums of harmonic numbers is the contour integral representation approach developed by Flajolet and Salvy \cite{Flajolet}. Due to some parity restrictions, some Euler sums can only be computed in pairs. For hyperharmonic number case,  Mez\H o and Dil \cite{mezo} considered the following infinite sum
$$
\sum_{n=1}^\infty \frac{h_n^{(r)}}{n^{m}} \quad (m\ge r+1, m \in \mathbb N),
$$
and showed that it could be reduced to infinite series involving the Hurwitz zeta function values. Later Dil and Boyadzhiev \cite{Dil1} extended this result to infinite series involving the multiple sums of the Hurwitz zeta function values.

The analytic continuation of the Riemann zeta function has been studied extensively, the reader can find its analytic continuation in any standard textbook of analytic number theory. On the contrary, the analytic continuation of Euler sums of hyperharmonic numbers has also attracted mathematician's interest. If we regard $\sum_{n=1}^\infty h_n^{(r)}/{n^{s}}$ as a complex function in variable $s$, i.e. we regard the hyperharmonic number as a weight in the infinite series, we can obtain more results in this direction. For example, Matsuoka \cite{Matsuoka}proved that $\sum_{n=1}^\infty h_n^{(1)}/{n^{s}}$ admits a meromorphic continuation to the whole complex plane. Kamano \cite{Kamano} proved that the complex variable function $\sum_{n=1}^\infty h_n^{(r)}/{n^{s}}$  could be meromorphically continued to the entire complex plane. In addition, the residue at each pole was also given.

As a natural consideration, Dil, Mez\H o and Cenkci \cite{Dil1} considered the Euler sums of the generalized hyperharmonic numbers
$
\zeta_{H^{(p,r)}}(m):= \sum_{n=1}^\infty H_n^{(p,r)}/{n^{m}}.
$
They proved that for positive integers $p, r$ and $m$ with $m>r$, $\zeta_{H^{(p,r)}}(m)$ could be reduced to infinite series of multiple sums of the Hurwitz zeta function values, which extended Dil-Boyadzhiev's result. For $r=1, 2, 3$, $\zeta_{H^{(p,r)}}(m)$ were also written explicitly by means of (multiple) zeta values. By computing a series of special polynomials, the author \cite{Rusen} obtain a nice formula for the generalized hyperharmonic numbers $H_n^{(p,r)}$. As a corollary, the author proved that $\zeta_{H^{(p,r)}}(m)$ could be expressed as linear combinations of classical Euler sums.

Apart from analytic properties, there are some more interesting combinatorial properties about the generalized hyperharmonic numbers. For instance, \" Om\" ur and Koparal \cite{omur} introduced the generalized hyperharmonic numbers $H_n^{(p,r)}$ from a combinatorial point of view, defined two $n\times n$ matrices $A_n$ and $B_n$ with $a_{i,j}=H_i^{(j,r)}$ and $b_{i,j}=H_i^{(p,j)}$, respectively, and gave some interesting factorizations and determinant properties of the matrices $A_n$ and $B_n$.

Motivated by Flajolet-Salvy's paper \cite{Flajolet} and Dil-Mez\H o-Cenkci's paper \cite{Dil}, the author \cite{LiRusen} introduced the notion of the generalized alternating hyperharmonic numbers
and proved that Euler sums of the generalized alternating hyperharmonic numbers (of type I) $H_n^{(p,r,1)}$ could be expressed in terms of linear combinations of classical (alternating) Euler sums.

Before going further, we introduce the definition of the generalized alternating hyperharmonic numbers $H_n^{(p,r,s_{1},s_{2})}$.
\begin{Definition}
For $p, r, n \in \mathbb N$ and $s_{1},s_{2} \in \mathbb N \cup \{0\}$ with $s_{1}+s_{2} \geq 1$, define the generalized alternating hyperharmonic numbers of type $(p,r,s_{1},s_{2})$ as
\begin{align}\label{deftrnt1}
{\scriptsize
H_n^{(p,r,s_{1},s_{2})}
:=\begin{cases}
\sum_{k=1}^{n} H_{k}^{(p,r-1,s_{1},s_{2})},& \text{$r-1\equiv 1,\cdots,s_{1}\pmod{s_{1}+s_{2}}$};\\
\sum_{k=1}^n (-1)^{k-1} H_{k}^{(p,r-1,s_{1},s_{2})},& \text{$r-1\equiv s_{1}+1,\cdots,s_{1}+s_{2}\pmod{s_{1}+s_{2}}$}.
\end{cases}\,
}
\end{align}
\end{Definition}

It is easy to verify that if $H_n^{(p,1,s_{1},s_{2})}=H_n^{(p)}$, then
\begin{align*}
H_n^{(p,r,s_{1},s_{2})}=\begin{cases}
H_n^{(p,r)},& \text{$s_{1}=1, s_{2}=0$};\\
H_n^{(p,r,1)},& \text{$s_{1}=0, s_{2}=1$}.
\end{cases}\,
\end{align*}

In this paper, we firstly investigate two special cases $H_n^{(p,r,2,1)}$ and $H_n^{(p,r,1,2)}$, and prove that Euler sums of these two numbers could be reduced to linear combinations of classical (alternating) Euler sums. Secondly, by using similar method, we obtain same results for the generalized alternating hyperharmonic numbers $H_n^{(p,r,s_{1},s_{2})}$. In addition, we also present several conjectures on these coefficients.

To be able to reach our goal, we recall truncated Faulhaber's formula on (alternating) sums of powers.

It is well known that the sum of powers of consecutive intergers $1^k+2^k+\cdots+n^k$ can be explicitly expressed in terms of Bernoulli numbers or Bernoulli polynomials. Faulhaber's formula can be written as
\begin{align}
\sum_{\ell=1}^{n}\ell^{k}&=\frac{1}{k+1}\sum_{j=0}^k \binom{k+1}{j}B_j^{+} n^{k+1-j}\label{ber}\\
                       &=\frac{1}{k+1}(B_{k+1}(n+1)-B_{k+1}(1))\quad\hbox{\cite{CFZ}}\,,
\label{ber1}
\end{align}
where Bernoulli numbers $B_n^{+}$ are determined by the recurrence formula
$$
\sum_{j=0}^k\binom{k+1}{j}B_j^{+}=k+1\quad (k\ge 0)
$$
or by the generating function
\begin{align}\label{defbernou}
\frac{t}{1-e^{-t}}=\sum_{n=0}^\infty B_n^{+}\frac{t^n}{n!}\,,
\end{align}
and  Bernoulli polynomials $B_n(x)$ are defined by the following generating function
$$
\frac{te^{xt}}{e^{t}-1}=\sum_{n=0}^\infty B_n(x)\frac{t^n}{n!}\,.
$$

We now provide explicit expressions of truncated Faulhaber's formula on (alternating) sums of powers.
\begin{Lem}{\cite{LiRusen}}\label{prop1}
For $m, n, t \in \mathbb N$ with $t \le n$, one has
\begin{align*}
\sum_{\ell=t}^{n}\ell^{m}=\sum_{j=1}^{m+1}c(m,j) n^j -\sum_{j=1}^{m+1}d(m,j) t^j-e(m,0)\,,
\end{align*}
where
\begin{align*}
&c(m,j)=\frac{1}{m+1} \binom{m+1}{m+1-j}B_{m+1-j}^{+}\quad (1 \le j \le m+1)\,,\notag\\
&d(m,j)=\frac{1}{m+1} \sum_{k=j-1}^{m}\binom{m+1}{m-k}B_{m-k}^{+}\binom{1+k}{j}(-1)^{1+k-j}\quad (1 \le j \le m+1)\,,\notag\\
&e(m,0)=\frac{1}{m+1} \sum_{k=0}^{m}\binom{m+1}{m-k}B_{m-k}^{+}(-1)^{1+k}\,.\notag\\
\end{align*}
\end{Lem}

\begin{Remark}
By considering $t=(1-e^{-t}) \cdot \sum_{n=0}^\infty B_n^{+}\frac{t^n}{n!}$, using Taylor expansion and comparing coefficients on both sides, we get that $e(m,0)=-\delta_{m 0}$, where $\delta_{m n}$ is the Kronecker delta, that is, $\delta_{m m}=1$, $\delta_{m n}=0$ for $m \neq n$.
\end{Remark}

\begin{Lem}{\cite{LiRusen}}\label{prop2}
For $m, n, t \in \mathbb N$ with $t \le n$, one has
\begin{align*}
\sum_{\ell=t}^{n}(-1)^{\ell-1}\ell^{m}
=\sum_{j=0}^{m}c_{1}(m,j)(-1)^{n-1}n^j+\sum_{j=0}^{m}d_{1}(m,j)(-1)^{t-1} t^j\notag\,,
\end{align*}
where
\begin{align*}
&c_{1}(m,j)=\frac{1}{2(m+1)} \sum_{k=0}^{m-j}\binom{m+1}{k}B_{k}^{+}2^{k}\binom{m+1-k}{j}(-1)^{m-k-j}\,,\notag\\
&d_{1}(m,j)=\frac{1}{2(m+1)} \sum_{k=j}^{m}\binom{k}{j}(-1)^{k-j}\sum_{x=0}^{m-k}\binom{m+1}{x}B_{x}^{+}2^{x}\binom{m+1-x}{k}(-1)^{m-x-k}\,.\notag\\
\end{align*}
\end{Lem}

\begin{Remark}
It seems that the following identities hold:
\begin{align*}
&c(m,j)=(-1)^{m+1-j} d(m,j)\quad (1 \le j \le m+1)\,,\notag\\
&c_{1}(m,j)=(-1)^{m-j} d_{1}(m,j)\quad (0 \le j \le m)\,.\notag\\
\end{align*}
Since these two identities are not essential in the proof of our main theorem, we leave it as a problem to the reader.
\end{Remark}

\section{Euler sums of generalized alternating hyperharmonic numbers $H_n^{(p,r,2,1)}$ }

In this section, we prove that Euler sums of the generalized alternating hyperharmonic numbers $H_n^{(p,r,2,1)}$ can be expressed in terms of linear combinations of classical (alternating) Euler sums. In order to reach our goal, we introduce some notations.

%Before going further in this direction, we introduce some notations.
\begin{Definition}\label{deftrnt4}
For $r, n, t \in \mathbb N$, define
\begin{align*}
T(r,2,1,n,t)=\begin{cases}
\sum_{k=t}^{n} T(r-1,2,1,k,t),& \text{$r-1\equiv 1,2\pmod{3}$};\\
\sum_{k=t}^n (-1)^{k-1} T(r-1,2,1,k,t),& \text{$r-1\equiv 3\pmod{3}$}.
\end{cases}
\end{align*}
\end{Definition}

\begin{Definition}\label{deftrnt5}
For $r \in \mathbb N$, define
\begin{align*}
f(r,2,1)=\begin{cases}
f(r-1,2,1)+1,& \text{$r-1\equiv 1,2,6\pmod{6}$};\\
f(r-1,2,1),& \text{$r-1\equiv 3,4,5\pmod{6}$}.
\end{cases}
\end{align*}
\end{Definition}

By direct computation, we have
\begin{align*}
&T(1,2,1,n,t)= 1\,,\\
&T(2,2,1,n,t)=n-t+1\,,\\
&T(3,2,1,n,t)=\frac{1}{2}n^{2}+(-t+\frac{3}{2})n+\frac{1}{2}t^{2}-\frac{3}{2}t+1\,,\\
&T(4,2,1,n,t)=\frac{1}{4}(-1)^{n-1}n^{2}+(-\frac{1}{2}t+1)(-1)^{n-1}n+\frac{1}{4}(-1)^{n-1}t^{2}\\
&\qquad \qquad \qquad \quad -(-1)^{n-1}t+\frac{7}{8}(-1)^{n-1}+\frac{1}{8}(-1)^{t-1}\,,\\
&T(5,2,1,n,t)=\frac{1}{8}(-1)^{n-1}n^{2}+\bigl(-\frac{1}{4}(-1)^{n-1}t+\frac{5}{8}(-1)^{n-1}+\frac{1}{8}(-1)^{t-1}\bigr)n\\ &\qquad \qquad \qquad \quad+\frac{1}{8}(-1)^{n-1}t^{2}-\bigl(\frac{5}{8}(-1)^{n-1}+\frac{1}{8}(-1)^{t-1}\bigr)t+\frac{11}{16}(-1)^{n-1}\\
&\qquad \qquad \qquad  \quad+\frac{5}{16}(-1)^{t-1}\,,\\
&T(6,2,1,n,t)=\frac{1}{16}\bigl((-1)^{n-1}+(-1)^{t-1}\bigr)(n^{2}+t^{2})-\frac{1}{8}\bigl((-1)^{n-1}+(-1)^{t-1}\bigr)nt\\
&\qquad \quad \quad \quad \quad \quad +\frac{3}{8}\bigl((-1)^{n-1}+(-1)^{t-1}\bigr)(n-t)+\frac{1}{2}\bigl((-1)^{n-1}+(-1)^{t-1}\bigr)\,,\\
&T(7,2,1,n,t)=\frac{1}{48}(n^{3}-t^{3})+\bigl(\frac{7}{32}+\frac{1}{32}(-1)^{n-1+t-1}\bigr)(n^{2}+t^{2})-\frac{1}{16}n^{2}t\\
&\qquad \qquad \qquad  \quad +\frac{1}{16}nt^{2}-nt\bigl(\frac{7}{16}+\frac{1}{16}(-1)^{n-1+t-1}\bigr)\\
&\qquad \qquad \qquad  \quad +\bigl(\frac{67}{96}+\frac{7}{32}(-1)^{n-1+t-1}\bigr)(n-t)+\frac{21}{32}+\frac{11}{32}(-1)^{n-1+t-1}\,,\\
&T(8,2,1,n,t)=\frac{1}{192}n^{4}+n^{3}\bigl(-\frac{1}{48}t+\frac{1}{12}\bigr)
+n^{2}\bigl(\frac{1}{32}t^{2}-\frac{1}{4}t+\frac{89}{192}\bigr)\\
&\qquad \qquad \qquad \quad +\frac{1}{64}(-1)^{n-1+t-1}n^{2} +n\bigl(-\frac{1}{48}t^{3}+\frac{1}{4}t^{2}-\frac{89}{96}t\bigr)\\
&\qquad \qquad \qquad \quad +n\bigl(-\frac{1}{32}(-1)^{n-1+t-1}t+\frac{25}{24}+\frac{1}{8}(-1)^{n-1+t-1}\bigr)\\
&\qquad \qquad \qquad \quad +\frac{1}{192}t^{4}-\frac{1}{12}t^{3}+\bigl(\frac{89}{192}+\frac{1}{64}(-1)^{n-1+t-1}\bigr)t^{2}
-\frac{25}{24}t\\
&\qquad \qquad \qquad \quad -\frac{1}{8}(-1)^{n-1+t-1}t+\frac{99}{128}+\frac{29}{128}(-1)^{n-1+t-1}\,.
\end{align*}

\begin{align*}
&f(1,2,1)=0,\quad f(2,2,1)=1,\quad f(3,2,1)=2,\quad f(4,2,1)=2,\\
&f(5,2,1)=2,\quad f(6,2,1)=2,\quad f(7,2,1)=3,\quad f(8,2,1)=4\,.
\end{align*}

Observing the above facts, we may assume
\begin{align*}
T(r,2,1,n,t)
&=\sum_{m=0}^{f(r,2,1)}\sum_{j=0}^{f(r,2,1)-m}\biggl(b(r,2,1,m,j,0)+b(r,2,1,m,j,1)(-1)^{n-1+t-1}\\
&\qquad+b(r,2,1,m,j,2)(-1)^{n-1}+b(r,2,1,m,j,3)(-1)^{t-1}\biggr)t^{j} n^{m}\,,
\end{align*}
with
\begin{align*}
b(r,2,1,m,j,k)=0\quad (m+j\geq f(r,2,1)+1, m\geq 0, j\geq 0, k=0,1,2,3)\,.
\end{align*}
We now prove the above assumption by mathematical induction.
\begin{Lem}\label{lemma1}
For $r, n, t \in \mathbb N$, we have
\begin{align*}
T(r,2,1,n,t)
&=\sum_{m=0}^{f(r,2,1)}\sum_{j=0}^{f(r,2,1)-m}\biggl(b(r,2,1,m,j,0)+b(r,2,1,m,j,1)(-1)^{n-1+t-1}\\
&\qquad+b(r,2,1,m,j,2)(-1)^{n-1}+b(r,2,1,m,j,3)(-1)^{t-1}\biggr)t^{j} n^{m}\,,
\end{align*}

where
\begin{align*}
&b(r,2,1,m,j,2)=b(r,2,1,m,j,3)=0\quad \bigl(r\equiv 1,2,3\pmod{6}\bigr)\,,\\
&b(r,2,1,m,j,0)=b(r,2,1,m,j,1)=0\quad \bigl(r\equiv 4,5,6\pmod{6}\bigr)\,,\\
&b(r,2,1,m,j,1)=0 \quad \bigl(r\equiv 1\pmod{6},\quad m,j\geq 0,\quad m+j\geq f(r,2,1)\bigr)\,,\\
&b(r,2,1,m,j,1)=0 \quad \bigl(r\equiv 2\pmod{6},\quad m,j\geq 0,\quad m+j\geq f(r,2,1)-1\bigr)\,,\\
&b(r,2,1,m,j,1)=0 \quad \bigl(r\equiv 3\pmod{6},\quad m,j\geq 0,\quad m+j\geq f(r,2,1)-2\bigr)\,,\\
&b(r,2,1,m,j,3)=0 \quad \bigl(r\equiv 4\pmod{6},\quad m,j\geq 0,\quad m+j\geq f(r,2,1)-1\bigr)\,,\\
&b(r,2,1,m,j,3)=0 \quad \bigl(r\equiv 5\pmod{6},\quad m,j\geq 0,\quad m+j\geq f(r,2,1)\bigr)\,,\\
&\sum_{m=0}^{f(r,2,1)}b(r,2,1,m,f(r,2,1)-m,0)=0 \quad \bigl(r\equiv 2,3\pmod{6}\bigr)\,,\\
&\sum_{m=0}^{f(r,2,1)-1}b(r,2,1,m,f(r,2,1)-1-m,0)=0 \quad \bigl(r\equiv 3\pmod{6}\bigr)\,,\\
\end{align*}
and for $k=0,1,2,3$, $b(r,2,1,m,j,k)$ satisfy the following recurrence relations:

Case $1$, if $r\equiv 1,2\pmod{6}$, then we have $f(r+1,2,1)=f(r,2,1)+1$,
\begin{align*}
&b(r+1,2,1,m,j,0)=\sum_{\ell=m-1}^{f(r,2,1)}b(r,2,1,\ell,j,0) c(\ell,m)\,,\\
&(1 \le m \le f(r,2,1)+1,\quad 0 \le j \le f(r,2,1)+1-m)\,;\\
&\quad b(r+1,2,1,0,j,0)\\
&=-\sum_{m=0}^{f(r,2,1)}\sum_{\substack{j_{1}+\ell=j\\ 0 \le j_{1} \le f(r,2,1)-m\\ 1 \le \ell \le m+1 }}b(r,2,1,m,j_{1},0) d(m,\ell)+b(r,2,1,0,j,0)\\
&\quad +\sum_{m=0}^{f(r,2,1)}\sum_{\substack{j_{1}+\ell=j\\ 0 \le j_{1} \le f(r,2,1)-m\\ 0 \le \ell \le m }}b(r,2,1,m,j_{1},1) d_{1}(m,\ell)\quad(0 \le j \le f(r,2,1)+1)\,;\\
&b(r+1,2,1,m,j,1)=\sum_{\ell=m}^{f(r,2,1)}b(r,2,1,\ell,j,1) c_{1}(\ell,m)\,,\\
&(0 \le m \le f(r,2,1),\quad 0 \le j \le f(r,2,1)-m)\,.
\end{align*}

Case $2$, if $r\equiv 3\pmod{6}$, then we have $f(r+1,2,1)=f(r,2,1)$,
\begin{align*}
&b(r+1,2,1,m,j,2)=\sum_{\ell=m}^{f(r,2,1)}b(r,2,1,\ell,j,0) c_{1}(\ell,m)\,,\\
&(0 \le m \le f(r,2,1),\quad 0 \le j \le f(r,2,1)-m)\,;\\
&b(r+1,2,1,m,j,3)=\sum_{\ell=m-1}^{f(r,2,1)}b(r,2,1,\ell,j,1) c(\ell,m)\,,\\
&(1 \le m \le f(r,2,1)+1,\quad 0 \le j \le f(r,2,1)+1-m)\,;\\
&\quad b(r+1,2,1,0,j,3)\\
&=\sum_{m=0}^{f(r,2,1)}\sum_{\substack{j_{1}+\ell=j\\ 0 \le j_{1} \le f(r,2,1)-m\\ 0 \le \ell \le m}}b(r,2,1,m,j_{1},0) d_{1}(m,\ell)+b(r,2,1,0,j,1)\\
&\quad -\sum_{m=0}^{f(r,2,1)}\sum_{\substack{j_{1}+\ell=j\\ 0 \le j_{1} \le f(r,2,1)-m\\ 1 \le \ell \le m+1 }}b(r,2,1,m,j_{1},1) d(m,\ell)\quad (0 \le j \le f(r,2,1)+1)\,.
\end{align*}

Case $3$, if $r\equiv 4,5\pmod{6}$, then we have $f(r+1,2,1)=f(r,2,1)$,
\begin{align*}
&b(r+1,2,1,m,j,2)=\sum_{\ell=m}^{f(r,2,1)}b(r,2,1,\ell,j,2) c_{1}(\ell,m)\,,\\
&(0 \le m \le f(r,2,1),\quad 0 \le j \le f(r,2,1)-m)\,;\\
&b(r+1,2,1,m,j,3)=\sum_{\ell=m-1}^{f(r,2,1)}b(r,2,1,\ell,j,3) c(\ell,m)\,,\\
&(1 \le m \le f(r,2,1)+1,\quad 0 \le j \le f(r,2,1)+1-m)\,;\\
&\quad b(r+1,2,1,0,j,3)\\
&=\sum_{m=0}^{f(r,2,1)}\sum_{\substack{j_{1}+\ell=j\\ 0 \le j_{1} \le f(r,2,1)-m\\ 0 \le \ell \le m}}b(r,2,1,m,j_{1},2) d_{1}(m,\ell)+b(r,2,1,0,j,3)\\
&\quad -\sum_{m=0}^{f(r,2,1)}\sum_{\substack{j_{1}+\ell=j\\ 0 \le j_{1} \le f(r,2,1)-m\\ 1 \le \ell \le m+1 }}b(r,2,1,m,j_{1},3) d(m,\ell)\quad (0 \le j \le f(r,2,1)+1)\,.
\end{align*}

Case $4$, if $r\equiv 6\pmod{6}$, then we have $f(r+1,2,1)=f(r,2,1)+1$,
\begin{align*}
&b(r+1,2,1,m,j,0)=\sum_{\ell=m-1}^{f(r,2,1)}b(r,2,1,\ell,j,2) c(\ell,m)\,,\\
&(1 \le m \le f(r,2,1)+1,\quad 0 \le j \le f(r,2,1)+1-m)\,;\\
&\quad b(r+1,2,1,0,j,0)\\
&=-\sum_{m=0}^{f(r,2,1)}\sum_{\substack{j_{1}+\ell=j\\ 0 \le j_{1} \le f(r,2,1)-m\\ 1 \le \ell \le m+1 }}b(r,2,1,m,j_{1},2) d(m,\ell)+b(r,2,1,0,j,2)\\
&\quad +\sum_{m=0}^{f(r,2,1)}\sum_{\substack{j_{1}+\ell=j\\ 0 \le j_{1} \le f(r,2,1)-m\\ 0 \le \ell \le m }}b(r,2,1,m,j_{1},3) d_{1}(m,\ell)\quad (0 \le j \le f(r,2,1)+1)\,;\\
&b(r+1,2,1,m,j,1)=\sum_{\ell=m}^{f(r,2,1)}b(r,2,1,\ell,j,3) c_{1}(\ell,m)\,,\\
&(0 \le m \le f(r,2,1),\quad 0 \le j \le f(r,2,1)-m)\,.
\end{align*}
The initial value is given by $T(1,2,1,n,t)= 1$.
\end{Lem}

\begin{proof}
The proof is done by induction on r. When $r=1$, the identity is
clear. Assume, then, the  identity has been proved for $1, 2, \cdots, r$. From the definition \ref{deftrnt4} of $T(r,2,1,n,t)$, we have

\begin{align*}
&\qquad T(r+1,2,1,n,t)\notag\\
&=\sum_{k_{1}=t}^{n} T_{1}(r,k_{1},t) \notag\\
&=\sum_{k_{1}=t}^{n} \sum_{m=0}^{f(r,2,1)}\sum_{j=0}^{f(r,2,1)-m} \biggl(b(r,2,1,m,j,0)+b(r,2,1,m,j,1)(-1)^{k_{1}-1+t-1}\notag\\
&\qquad+b(r,2,1,m,j,2)(-1)^{k_{1}-1}+b(r,2,1,m,j,3)(-1)^{t-1}\biggr)t^{j} k_{1}^{m}\notag\\
&=S_{1}+S_{2}+S_{3}+S_{4}\,, \notag
\end{align*}
where
\begin{align*}
&S_{1}=\sum_{m=0}^{f(r,2,1)}\sum_{j=0}^{f(r,2,1)-m}b(r,2,1,m,j,0)
t^{j}\sum_{k_{1}=t}^{n}k_{1}^{m}\,,\\
&S_{2}=\sum_{m=0}^{f(r,2,1)}\sum_{j=0}^{f(r,2,1)-m}b(r,2,1,m,j,1)
(-1)^{t-1}t^{j}\sum_{k_{1}=t}^{n}(-1)^{k_{1}-1}k_{1}^{m}\,,\\
&S_{3}=\sum_{m=0}^{f(r,2,1)}\sum_{j=0}^{f(r,2,1)-m}b(r,2,1,m,j,2)
t^{j}\sum_{k_{1}=t}^{n}(-1)^{k_{1}-1}k_{1}^{m}\,,\\
&S_{4}=\sum_{m=0}^{f(r,2,1)}\sum_{j=0}^{f(r,2,1)-m}b(r,2,1,m,j,1)
(-1)^{t-1}t^{j}\sum_{k_{1}=t}^{n}k_{1}^{m}\,,
\end{align*}
and
\begin{align*}
&\qquad T(r+1,2,1,n,t)\notag\\
&=\sum_{k_{1}=t}^{n}(-1)^{k_{1}-1}T_{1}(r,k_{1},t) \notag\\
&=\sum_{k_{1}=t}^{n}(-1)^{k_{1}-1} \sum_{m=0}^{f(r,2,1)}\sum_{j=0}^{f(r,2,1)-m} \biggl(b(r,2,1,m,j,0)+b(r,2,1,m,j,1)(-1)^{k_{1}-1+t-1}\notag\\
&\qquad+b(r,2,1,m,j,2)(-1)^{k_{1}-1}+b(r,2,1,m,j,3)(-1)^{t-1}\biggr)t^{j} k_{1}^{m}\notag\\
&=S_{5}+S_{6}+S_{7}+S_{8}\,, \notag
\end{align*}
where
\begin{align*}
&S_{5}=\sum_{m=0}^{f(r,2,1)}\sum_{j=0}^{f(r,2,1)-m}b(r,2,1,m,j,0)
t^{j}\sum_{k_{1}=t}^{n}k_{1}^{m}\,,\\
&S_{6}=\sum_{m=0}^{f(r,2,1)}\sum_{j=0}^{f(r,2,1)-m}b(r,2,1,m,j,1)
(-1)^{t-1}t^{j}\sum_{k_{1}=t}^{n}(-1)^{k_{1}-1}k_{1}^{m}\,,\\
&S_{7}=\sum_{m=0}^{f(r,2,1)}\sum_{j=0}^{f(r,2,1)-m}b(r,2,1,m,j,2)
t^{j}\sum_{k_{1}=t}^{n}(-1)^{k_{1}-1}k_{1}^{m}\,,\\
&S_{8}=\sum_{m=0}^{f(r,2,1)}\sum_{j=0}^{f(r,2,1)-m}b(r,2,1,m,j,1)
(-1)^{t-1}t^{j}\sum_{k_{1}=t}^{n}k_{1}^{m}\,.
\end{align*}

Now we prove the above recurrence relations in $6$ steps.

Step $1$: If $r\equiv 1\pmod{6}$, then $f(r+1,2,1)=f(r,2,1)+1$. From the definition \ref{deftrnt4}, we have
$$
T(r+1,2,1,n,t)=S_{1}+S_{2}.
$$
With the help of Lemmas \ref{prop1} and \ref{prop2}, comparing the coefficients of $t^{j}n^m$ gives the recurrence relations stated in Case $1$. The only facts that need to be verified are
$$
b(r+1,2,1,m,j,1)=0 \quad \bigl(r+1\equiv 2\pmod{6},\quad m,j\geq 0, \quad m+j\geq f(r,2,1)\bigr)
$$
and
$$
\sum_{m=0}^{f(r+1,2,1)}b(r+1,2,1,m,f(r+1,2,1)-m,0)=0 \quad \bigl(r+1\equiv 2\pmod{6}\bigr)\,.
$$

By the induction hypothesis, we have
$$
b(r,2,1,m,j,1)=0 \quad \bigl(r\equiv 1\pmod{6},\quad m,j\geq 0,\quad m+j\geq f(r,2,1)\bigr).
$$
Since
$$
b(r+1,2,1,m,j,1)=\sum_{\ell=m}^{f(r,2,1)}b(r,2,1,\ell,j,1) c_{1}(\ell,m),
$$
then we get that
$$
b(r+1,2,1,m,j,1)=0 \quad \bigl(r+1\equiv 2\pmod{6},\quad m,j\geq 0,\quad m+j\geq f(r,2,1)\bigr)\,.
$$

Since
\begin{align*}
&b(r+1,2,1,m,f(r,2,1)+1-m,0)=\frac{1}{m} b(r,2,1,m-1,f(r,2,1)+1-m,0)\\
&(1 \le m \le f(r,2,1)+1)\,
\end{align*}
and
$$
b(r+1,2,1,0,f(r,2,1)+1,0)=-\frac{1}{m+1}\sum_{m=0}^{f(r,2,1)}b(r,2,1,m,f(r,2,1)-m,0),
$$
then we have
$$
\sum_{m=0}^{f(r+1,2,1)}b(r+1,2,1,m,f(r+1,2,1)-m,0)=0 \quad \bigl(r+1\equiv 2\pmod{6}\bigr)\,.
$$

Step $2$: If $r\equiv 2\pmod{6}$, then $f(r+1,2,1)=f(r,2,1)+1$. From the definition \ref{deftrnt4}, we have
$$
T(r+1,2,1,n,t)=S_{1}+S_{2}.
$$
With the help of Lemmas \ref{prop1} and \ref{prop2}, comparing the coefficients of $t^{j}n^m$ gives the recurrence relations stated in Case $1$. The only facts that need to be verified are
$$
b(r+1,2,1,m,j,1)=0 \quad \bigl(r+1\equiv 3\pmod{6},\quad m,j\geq 0, \quad m+j\geq f(r,2,1)-1\bigr)\,,
$$

$$
\sum_{m=0}^{f(r+1,2,1)}b(r+1,2,1,m,f(r+1,2,1)-m,0)=0 \quad \bigl(r+1\equiv 3\pmod{6}\bigr)\,,
$$
and
$$
\sum_{m=0}^{f(r+1,2,1)-1}b(r+1,2,1,m,f(r+1,2,1)-1-m,0)=0 \quad \bigl(r+1\equiv 3\pmod{6}\bigr)\,.
$$
By the induction hypothesis, we have
$$
b(r,2,1,m,j,1)=0 \quad \bigl(r\equiv 2\pmod{6},\quad m,j\geq 0,\quad m+j\geq f(r,2,1)-1\bigr).
$$
Since
$$
b(r+1,2,1,m,j,1)=\sum_{\ell=m}^{f(r,2,1)}b(r,2,1,\ell,j,1) c_{1}(\ell,m),
$$
then we get that
$$
b(r+1,2,1,m,j,1)=0 \quad \bigl(m,j\geq 0,\quad m+j\geq f(r,2,1)-1\bigr)\,.
$$
Similar with Step $1$, we also have
$$
\sum_{m=0}^{f(r+1,2,1)}b(r+1,2,1,m,f(r+1,2,1)-m,0)=0 \quad \bigl(r+1\equiv 3\pmod{6}\bigr).
$$
By using the recursive formula for $b(r+1,2,1,m,f(r+1,2,1)-1-m,0)$, we get that
\begin{align*}
&\quad \sum_{m=0}^{f(r+1,2,1)-1}b(r+1,2,1,m,f(r+1,2,1)-1-m,0)\\
&=\sum_{m=1}^{f(r,2,1)}b(r+1,2,1,m,f(r,2,1)-m,0)+b(r+1,2,1,0,f(r,2,1),0)\,\\
&=\sum_{m=1}^{f(r,2,1)}\Bigl(\frac{1}{m}b(r,2,1,m-1,f(r,2,1)-m,0)+\frac{1}{2}b(r,2,1,m,f(r,2,1)-m,0)\Bigr)\,\\
&\quad +b(r,2,1,0,f(r,2,1),1)-\sum_{m=0}^{f(r,2,1)-1}\frac{1}{m+1}b(r,2,1,m,f(r,2,1)-1-m,0)\,\\
&\quad +\sum_{m=1}^{f(r,2,1)}\frac{1}{2}b(r,2,1,m,f(r,2,1)-m,0)\,\\
&=\sum_{m=0}^{f(r,2,1)}b(r,2,1,m,f(r,2,1)-m,0)\,\\
&=0 \quad \bigl(r+1\equiv 3\pmod{6}\bigr)\,.
\end{align*}

Step $3$: If $r\equiv 3\pmod{6}$, then $f(r+1,2,1)=f(r,2,1)$. From the definition \ref{deftrnt4}, we have
$$
T(r+1,2,1,n,t)=S_{5}+S_{6}.
$$
With the help of Lemmas \ref{prop1} and \ref{prop2}, comparing the coefficients of $t^{j}n^m$ gives the recurrence relations stated in Case $2$. The only facts that need to be verified are
$$
b(r+1,2,1,m,j,3)=0 \quad \bigl(r+1\equiv 4\pmod{6},\quad m,j\geq 0, \quad m+j\geq f(r,2,1)-1\bigr).
$$
By the induction hypothesis, we have
$$
b(r,2,1,m,j,1)=0 \quad \bigl(r\equiv 3\pmod{6},\quad m,j\geq 0,\quad m+j\geq f(r,2,1)-2\bigr).
$$
If $m\geq 1, j\geq 0$, $m+j \geq f(r,2,1)-1$, since
$$
b(r+1,2,1,m,j,3)=\sum_{\ell=m-1}^{f(r,2,1)}b(r,2,1,\ell,j,1) c(\ell,m),
$$
then we get that
$$
b(r+1,2,1,m,j,3)=0\,.
$$
For $j\geq f(r,2,1)-1$, we have the following three identities:
\begin{align*}
b(r+1,2,1,0,f(r,2,1)+1,3)
&=-\sum_{m=0}^{f(r,2,1)}\sum_{\substack{j_{1}+\ell=f(r,2,1)+1\\ 0 \le j_{1} \le f(r,2,1)-m\\ 1 \le \ell \le m+1 }}b(r,2,1,m,j_{1},1) d(m,\ell)\,\\
&=-\sum_{m=0}^{f(r,2,1)}b(r,2,1,m,f(r,2,1)-m,1) d(m,m+1)\,\\
&=0\quad \bigl(r+1\equiv 4\pmod{6}\bigr)\,.
\end{align*}
\begin{align*}
&\qquad b(r+1,2,1,0,f(r,2,1),3)\\
&=\sum_{m=0}^{f(r,2,1)}\sum_{\substack{j_{1}+\ell=f(r,2,1)\\ 0 \le j_{1} \le f(r,2,1)-m\\ 0 \le \ell \le m}}b(r,2,1,m,j_{1},0) d_{1}(m,\ell)\\
&\quad -\sum_{m=0}^{f(r,2,1)}\sum_{\substack{j_{1}+\ell=f(r,2,1)\\ 0 \le j_{1} \le f(r,2,1)-m\\ 1 \le \ell \le m+1 }}b(r,2,1,m,j_{1},1) d(m,\ell)+b(r,2,1,0,f(r,2,1),1)\,\\
&=\frac{1}{2}\sum_{m=0}^{f(r,2,1)}b(r,2,1,m,f(r,2,1)-m,0)\\
&=0\quad \bigl(r+1\equiv 4\pmod{6}\bigr)\,.
\end{align*}
\begin{align*}
&\qquad b(r+1,2,1,0,f(r,2,1)-1,3)\\
&=\sum_{m=0}^{f(r,2,1)}\sum_{\substack{j_{1}+\ell=f(r,2,1)-1\\ 0 \le j_{1} \le f(r,2,1)-m\\ 0 \le \ell \le m}}b(r,2,1,m,j_{1},0) d_{1}(m,\ell)\\
&\quad -\sum_{m=0}^{f(r,2,1)}\sum_{\substack{j_{1}+\ell=f(r,2,1)-1\\ 0 \le j_{1} \le f(r,2,1)-m\\ 1 \le \ell \le m+1 }}b(r,2,1,m,j_{1},1) d(m,\ell)+b(r,2,1,0,f(r,2,1)-1,1)\,\\
&=\frac{1}{2}\sum_{m=0}^{f(r,2,1)-1}b(r,2,1,m,f(r,2,1)-1-m,0)\\
&\quad -\frac{1}{4}\sum_{m=1}^{f(r,2,1)} m b(r,2,1,m,f(r,2,1)-m,0)\\
&=-\frac{1}{4}\sum_{m=1}^{f(r-1,2,1)+1} m b(r-1,2,1,m-1,f(r-1,2,1)+1-m,0)c(m-1,m)\\
&=-\frac{1}{4}\sum_{m=1}^{f(r-1,2,1)+1} b(r-1,2,1,m-1,f(r-1,2,1)+1-m,0)\\
&=0\quad \bigl(r+1\equiv 4\pmod{6}\bigr)\,.
\end{align*}

Step $4$: If $r\equiv 4\pmod{6}$, then $f(r+1,2,1)=f(r,2,1)$. From the definition \ref{deftrnt4}, we have
$$
T(r+1,2,1,n,t)=S_{3}+S_{4}.
$$
With the help of Lemmas \ref{prop1} and \ref{prop2}, comparing the coefficients of $t^{j}n^m$ gives the recurrence relations stated in Case $3$. The only facts that need to be verified are
$$
b(r+1,2,1,m,j,3)=0 \quad \bigl(r+1\equiv 5\pmod{6},\quad m,j\geq 0, \quad m+j\geq f(r,2,1)\bigr).
$$
By the induction hypothesis, we have
$$
b(r,2,1,m,j,3)=0 \quad \bigl(r\equiv 4\pmod{6},\quad m,j\geq 0,\quad m+j\geq f(r,2,1)-1\bigr).
$$
If $m\geq 1, j\geq 0$, $m+j \geq f(r,2,1)$, since
$$
b(r+1,2,1,m,j,3)=\sum_{\ell=m-1}^{f(r,2,1)}b(r,2,1,\ell,j,3) c(\ell,m)\,,
$$
then we get that
$$
b(r+1,2,1,m,j,3)=0\,.
$$

For $j\geq f(r,2,1)$, we have the following two identities:
\begin{align*}
b(r+1,2,1,0,f(r,2,1)+1,3)
&=-\sum_{m=0}^{f(r,2,1)}\sum_{\substack{j_{1}+\ell=f(r,2,1)+1\\ 0 \le j_{1} \le f(r,2,1)-m\\ 1 \le \ell \le m+1 }}b(r,2,1,m,j_{1},3) d(m,\ell)\,\\
&=-\sum_{m=0}^{f(r,2,1)}\frac{1}{m+1}b(r,2,1,m, f(r,2,1)-m,3)\,\\
&=0\quad \bigl(r+1\equiv 5\pmod{6}\bigr)\,.
\end{align*}
\begin{align*}
b(r+1,2,1,0,f(r,2,1),3)
&=\frac{1}{2}\sum_{m=0}^{f(r,2,1)}b(r,2,1,m,f(r,2,1)-m,2)\\
&=\frac{1}{4}\sum_{m=0}^{f(r-1,2,1)}b(r-1,2,1,m,f(r-1,2,1)-m,0)\\
&=0\quad \bigl(r+1\equiv 5\pmod{6}\bigr)\,.
\end{align*}

Step $5$: If $r\equiv 5\pmod{6}$, then $f(r+1,2,1)=f(r,2,1)$. From the definition \ref{deftrnt4}, we have
$$
T(r+1,2,1,n,t)=S_{3}+S_{4}.
$$
With the help of Lemmas \ref{prop1} and \ref{prop2}, comparing the coefficients of $t^{j}n^m$ gives the recurrence relations stated in Case $3$. The only facts that need to be verified are
$$
b(r+1,2,1,m,j,3)=0 \quad \bigl(r+1\equiv 6\pmod{6},\quad m,j\geq 0, \quad m+j\geq f(r,2,1)+1\bigr).
$$
By the induction hypothesis, we have
$$
b(r,2,1,m,j,3)=0 \quad \bigl(r\equiv 5\pmod{6},\quad m,j\geq 0,\quad m+j\geq f(r,2,1)\bigr).
$$
If $m\geq 1, j\geq 0$, $m+j \geq f(r,2,1)+1$, since
$$
b(r+1,2,1,m,j,3)=\sum_{\ell=m-1}^{f(r,2,1)}b(r,2,1,\ell,j,3) c(\ell,m)\,,
$$
then we get that
$$
b(r+1,2,1,m,j,3)=0\,.
$$

For $j\geq f(r,2,1)+1$, we have the following identity:
\begin{align*}
b(r+1,2,1,0,f(r,2,1)+1,3)
&=-\sum_{m=0}^{f(r,2,1)}\sum_{\substack{j_{1}+\ell=f(r,2,1)+1\\ 0 \le j_{1} \le f(r,2,1)-m\\ 1 \le \ell \le m+1 }}b(r,2,1,m,j_{1},3) d(m,\ell)\,\\
&=-\sum_{m=0}^{f(r,2,1)}\frac{1}{m+1}b(r,2,1,m, f(r,2,1)-m,3)\,\\
&=0\quad \bigl(r+1\equiv 6\pmod{6}\bigr)\,.
\end{align*}

Step $6$: If $r\equiv 6\pmod{6}$, then $f(r+1,2,1)=f(r,2,1)+1$. From the definition \ref{deftrnt4}, we have
$$
T(r+1,2,1,n,t)=S_{7}+S_{8}.
$$
With the help of Lemmas \ref{prop1} and \ref{prop2}, comparing the coefficients of $t^{j}n^m$ gives the recurrence relations stated in Case $4$. The only facts that need to be verified are
$$
b(r+1,2,1,m,j,1)=0 \quad \bigl(r+1\equiv 1\pmod{6},\quad m,j\geq 0, \quad m+j\geq f(r,2,1)+1\bigr).
$$
Notice that the recursive formula for $b(r+1,2,1,m,j,1)=0$ is valid only when $0 \le m \le f(r,2,1),\quad 0 \le j \le f(r,2,1)-m$, we have completed the proof.
\end{proof}

Now we are able to prove our main theorem of this section.
\begin{theorem}\label{maintheorem}
Let $r,p \in \mathbb N$ and $q$ be a positive real number with $q \geq r+1$, we have,
\begin{align*}
\sum_{n=1}^\infty \frac{H_n^{(p,r,2,1)}}{n^{q}}
&=\sum_{m=0}^{f(r,2,1)}\sum_{j=0}^{f(r,2,1)-m}
\biggl(b(r,2,1,j,m,0)S_{p-m,q-j}^{+,+}+b(r,2,1,j,m,1)S_{p-m,q-j}^{-,-}\\
&\qquad+b(r,2,1,j,m,2)S_{p-m,q-j}^{+,-}+b(r,2,1,j,m,3)S_{p-m,q-j}^{-,+}\biggr).
\end{align*}
Therefore Euler sums of the generalized alternating hyperharmonic numbers $H_n^{(p,r,2,1)}$ can be expressed in terms of linear combinations of classical (alternating) Euler sums.
\end{theorem}
\begin{proof}
The convergence of $\sum_{n=1}^\infty \frac{H_n^{(p,r,2,1)}}{n^{q}}$ is guaranteed by $\vert {H_n^{(p,r,2,1)}} \vert  \le H_n^{(p,r)}$ and the convergence of $\zeta_{H^{(p,r)}}(q)$ (see \cite{Dil}). Note that
\begin{align*}
\sum_{n=1}^\infty \frac{H_n^{(p,r,2,1)}}{n^{q}}
&=\sum_{n=1}^\infty \frac{1}{n^{q}}\sum_{t=1}^{n}\frac{1}{t^{p}}T(r,2,1,n,t)\,,
\end{align*}
with the help of Lemma \ref{lemma1}, we get the desired result.
\end{proof}

\section{Euler sums of generalized alternating hyperharmonic numbers $H_n^{(p,r,1,2)}$}

In this section, we prove that Euler sums of the generalized alternating hyperharmonic numbers $H_n^{(p,r,1,2)}$ can be expressed in terms of linear combinations of classical (alternating) Euler sums. In order to reach our goal, we introduce some notations.

%Before going further in this direction, we introduce some notations.
\begin{Definition}\label{deftrnt6}
For $r, n, t \in \mathbb N$, define
\begin{align*}
T(r,1,2,n,t)=\begin{cases}
\sum_{k=t}^{n} T(r-1,1,2,k,t),& \text{$r-1\equiv 1\pmod{3}$};\\
\sum_{k=t}^n (-1)^{k-1} T(r-1,1,2,k,t),& \text{$r-1\equiv 2,3\pmod{3}$}.
\end{cases}
\end{align*}
\end{Definition}

\begin{Definition}\label{deftrnt7}
For $r \in \mathbb N$, define
\begin{align*}
f(r,1,2)=\begin{cases}
f(r-1,1,2)+1,& \text{$r-1\equiv 1,3\pmod{3}$};\\
f(r-1,1,2),& \text{$r-1\equiv 2\pmod{3}$}.
\end{cases}
\end{align*}
\end{Definition}

By direct computation, we have
\begin{align*}
&T(1,1,2,n,t)= 1\,,\\
&T(2,1,2,n,t)=n-t+1\,,\\
&T(3,1,2,n,t)=\frac{1}{2}(-1)^{n-1}(n-t)+\frac{3}{4}(-1)^{n-1}+\frac{1}{4}(-1)^{t-1}\,,\\
&T(4,1,2,n,t)=\frac{1}{4}n^{2}+(-\frac{1}{2}t+1)n+\frac{1}{4}t^{2}-t+\frac{7}{8}+\frac{1}{8}(-1)^{n-1+t-1}\,,\\
&T(5,1,2,n,t)=\frac{1}{12}n^{3}+\bigl(-\frac{1}{4}t+\frac{5}{8}\bigr)n^{2}
+\bigl(\frac{1}{4}t^{2}-\frac{5}{4}t+\frac{17}{12}\bigr)n-\frac{1}{12}t^{3}+\frac{5}{8}t^{2}\\
&\qquad \qquad \quad \quad \quad-\frac{17}{12}t+\frac{15}{16}+\frac{1}{16}(-1)^{n-1+t-1}\,,\\
&T(6,1,2,n,t)=\frac{1}{24}(-1)^{n-1}n^{3}+(-\frac{1}{8}t+\frac{3}{8})(-1)^{n-1}n^{2}
+(\frac{1}{8}t^{2}-\frac{3}{4}t)(-1)^{n-1}n\\
&\qquad \qquad \quad \quad \quad  +\bigl(\frac{49}{48}(-1)^{n-1}+\frac{1}{16}(-1)^{t-1}\bigr)n-\frac{1}{24}(-1)^{n-1}t^{3}
+\frac{3}{8}(-1)^{n-1}t^{2}\\
&\qquad \qquad \quad \quad \quad-\bigl(\frac{49}{48}(-1)^{n-1}+\frac{1}{16}(-1)^{t-1}\bigr)t+\frac{13}{16}(-1)^{n-1}+\frac{3}{16}(-1)^{t-1}\,,\\
&T(7,1,2,n,t)=\frac{1}{96}n^{4}+\bigl(-\frac{1}{24}t+\frac{7}{48}\bigr)n^{3}
+\bigl(\frac{1}{16}t^{2}-\frac{7}{16}t+\frac{17}{24}\bigr)n^{2}-\frac{1}{24}t^{3}n\\
&\qquad \qquad \qquad \quad +\bigl(\frac{7}{16}t^{2}-\frac{17}{12}t+\frac{133}{96}+\frac{1}{32}(-1)^{n-1+t-1}\bigr)n +\frac{1}{96}t^{4}-\frac{7}{48}t^{3}\\
&\qquad \qquad \qquad \quad +\frac{17}{24}t^{2}-\bigl(\frac{133}{96}+\frac{1}{32}(-1)^{n-1+t-1}\bigr)t+\frac{57}{64}+\frac{7}{64}(-1)^{n-1+t-1}\,.
\end{align*}

\begin{align*}
&f(1,1,2)=0,\quad f(2,1,2)=1,\quad f(3,1,2)=1,\quad f(4,1,2)=2,\\
&f(5,1,2)=3,\quad f(6,1,2)=3,\quad f(7,1,2)=4\,.
\end{align*}

We now present main results of this section. We omit their proofs which are similar to those of Lemma \ref{lemma1} and Theorem \ref{maintheorem}.

\begin{Lem}\label{lemma12}
For $r, n, t \in \mathbb N$, we have
\begin{align*}
T(r,1,2,n,t)
&=\sum_{m=0}^{f(r,1,2)}\sum_{j=0}^{f(r,1,2)-m}\biggl(b(r,1,2,m,j,0)+b(r,1,2,m,j,1)(-1)^{n-1+t-1}\\
&\qquad+b(r,1,2,m,j,2)(-1)^{n-1}+b(r,1,2,m,j,3)(-1)^{t-1}\biggr)t^{j} n^{m}\,,
\end{align*}

where
\begin{align*}
&b(r,1,2,m,j,2)=b(r,1,2,m,j,3)=0\quad \bigl(r\equiv 1,2\pmod{3}\bigr)\,,\\
&b(r,1,2,m,j,0)=b(r,1,2,m,j,1)=0\quad \bigl(r\equiv 3\pmod{3}\bigr)\,,\\
&b(r,1,2,m,j,1)=0 \quad \bigl(r\equiv 1\pmod{3},\quad m,j\geq 0,\quad m+j\geq f(r,1,2)\bigr)\,,\\
&b(r,1,2,m,j,1)=0 \quad \bigl(r\equiv 2\pmod{3},\quad m,j\geq 0,\quad m+j\geq f(r,1,2)-1\bigr)\,,
\end{align*}
and for $k=0,1,2,3$, $b(r,1,2,m,j,k)$ satisfy the following recurrence relations:\\
Case $1^{\prime}$, if $r\equiv 1\pmod{3}$, then we have $f(r+1,1,2)=f(r,1,2)+1$,
\begin{align*}
&b(r+1,1,2,m,j,0)=\sum_{\ell=m-1}^{f(r,1,2)}b(r,1,2,\ell,j,0) c(\ell,m)\,,\\
&(1 \le m \le f(r,1,2)+1,\quad 0 \le j \le f(r,1,2)+1-m)\,;\\
&\quad b(r+1,1,2,0,j,0)\\
&=-\sum_{m=0}^{f(r,1,2)}\sum_{\substack{j_{1}+\ell=j\\ 0 \le j_{1} \le f(r,1,2)-m\\ 1 \le \ell \le m+1 }}b(r,1,2,m,j_{1},0) d(m,\ell)+b(r,1,2,0,j,0)\\
&\quad +\sum_{m=0}^{f(r,1,2)}\sum_{\substack{j_{1}+\ell=j\\ 0 \le j_{1} \le f(r,1,2)-m\\ 0 \le \ell \le m }}b(r,1,2,m,j_{1},1) d_{1}(m,\ell)\quad (0 \le j \le f(r,1,2)+1)\,;\\
&b(r+1,1,2,m,j,1)=\sum_{\ell=m}^{f(r,1,2)}b(r,1,2,\ell,j,1) c_{1}(\ell,m)\,,\\
&(0 \le m \le f(r,1,2),\quad 0 \le j \le f(r,1,2)-m)\,.
\end{align*}

Case $2^{\prime}$, if $r\equiv 2\pmod{3}$, then we have $f(r+1,2,1)=f(r,2,1)$,
\begin{align*}
&b(r+1,1,2,m,j,2)=\sum_{\ell=m}^{f(r,1,2)}b(r,1,2,\ell,j,0) c_{1}(\ell,m)\,,\\
&(0 \le m \le f(r,1,2),\quad 0 \le j \le f(r,1,2)-m)\,;\\
&b(r+1,1,2,m,j,3)=\sum_{\ell=m-1}^{f(r,1,2)}b(r,1,2,\ell,j,1) c(\ell,m)\,,\\
&(1 \le m \le f(r,1,2)+1,\quad 0 \le j \le f(r,1,2)+1-m)\,;\\
&\quad b(r+1,1,2,0,j,3)\\
&=\sum_{m=0}^{f(r,1,2)}\sum_{\substack{j_{1}+\ell=j\\ 0 \le j_{1} \le f(r,1,2)-m\\ 0 \le \ell \le m}}b(r,1,2,m,j_{1},0) d_{1}(m,\ell)+b(r,1,2,0,j,1)\\
&\quad -\sum_{m=0}^{f(r,1,2)}\sum_{\substack{j_{1}+\ell=j\\ 0 \le j_{1} \le f(r,1,2)-m\\ 1 \le \ell \le m+1 }}b(r,1,2,m,j_{1},1) d(m,\ell)\quad (0 \le j \le f(r,1,2)+1)\,.
\end{align*}

Case $3^{\prime}$, if $r\equiv 3\pmod{3}$, then we have $f(r+1,1,2)=f(r,1,2)+1$,
\begin{align*}
&b(r+1,1,2,m,j,0)=\sum_{\ell=m-1}^{f(r,1,2)}b(r,1,2,\ell,j,2) c(\ell,m)\,,\\
&(1 \le m \le f(r,1,2)+1,\quad 0 \le j \le f(r,1,2)+1-m)\,;\\
&\quad b(r+1,1,2,0,j,0)\\
&=-\sum_{m=0}^{f(r,1,2)}\sum_{\substack{j_{1}+\ell=j\\ 0 \le j_{1} \le f(r,1,2)-m\\ 1 \le \ell \le m+1 }}b(r,1,2,m,j_{1},2) d(m,\ell)+b(r,1,2,0,j,2)\\
&\quad +\sum_{m=0}^{f(r,1,2)}\sum_{\substack{j_{1}+\ell=j\\ 0 \le j_{1} \le f(r,1,2)-m\\ 0 \le \ell \le m }}b(r,1,2,m,j_{1},3) d_{1}(m,\ell)\quad (0 \le j \le f(r,1,2)+1)\,;\\
&b(r+1,1,2,m,j,1)=\sum_{\ell=m}^{f(r,1,2)}b(r,1,2,\ell,j,3) c_{1}(\ell,m)\,,\\
&(0 \le m \le f(r,1,2),\quad 0 \le j \le f(r,1,2)-m)\,.
\end{align*}
The initial value is given by $T(1,1,2,n,t)= 1$.
\end{Lem}

\begin{theorem}\label{maintheorem12}
Let $r,p \in \mathbb N$ and $q$ be a positive real number with $q \geq r+1$, we have,
\begin{align*}
\sum_{n=1}^\infty \frac{H_n^{(p,r,1,2)}}{n^{q}}
&=\sum_{m=0}^{f(r,1,2)}\sum_{j=0}^{f(r,1,2)-m}
\biggl(b(r,1,2,j,m,0)S_{p-m,q-j}^{+,+}+b(r,1,2,j,m,1)S_{p-m,q-j}^{-,-}\\
&\qquad+b(r,1,2,j,m,2)S_{p-m,q-j}^{+,-}+b(r,1,2,j,m,3)S_{p-m,q-j}^{-,+}\biggr).
\end{align*}
Therefore Euler sums of the generalized alternating hyperharmonic numbers $H_n^{(p,r,1,2)}$ can be expressed in terms of linear combinations of classical (alternating) Euler sums.
\end{theorem}
\begin{proof}
The proof is similar to that of Theorem \ref{maintheorem}, so we omit it.
\end{proof}

\section{Euler sums of generalized alternating hyperharmonic numbers $H_n^{(p,r,s_{1},s_{2})}$}

In this section, we prove that Euler sums of the generalized alternating hyperharmonic numbers $H_n^{(p,r,s_{1},s_{2})}$ can be expressed in terms of linear combinations of classical (alternating) Euler sums. In order to reach our goal, we introduce some notations.

%Before going further in this direction, we introduce some notations.
\begin{Definition}\label{deftrnt8}
For $r, n, t \in \mathbb N$, define
\begin{align*}
{\scriptsize
T(r,s_{1},s_{2},n,t)=\begin{cases}
\sum_{k=t}^{n} T(r-1,s_{1},s_{2},k,t),& \text{$r-1\equiv 1,\cdots,s_{1}\pmod{s_{1}+s_{2}}$};\\
\sum_{k=t}^n (-1)^{k-1} T(r-1,s_{1},s_{2},k,t),& \text{$r-1\equiv s_{1}+1,\cdots,s_{1}+s_{2}\pmod{s_{1}+s_{2}}$}.
\end{cases}
}
\end{align*}
\end{Definition}

\begin{Definition}\label{deftrnt9}
Let $f(1,s_{1},s_{2})=0$. For $r \in \mathbb N$, define $f(r,s_{1},s_{2})$ recursively as:\\
Case $s_{2}$=odd,
if $r-1\equiv 1, 2,\cdots, s_{1}\pmod{2(s_{1}+s_{2})}$, then
$$
f(r,s_{1},s_{2})=f(r-1,s_{1},s_{2})+1;
$$
if $r-1\equiv s_{1}+1, s_{1}+3,\cdots, s_{1}+s_{2}\pmod{2(s_{1}+s_{2})}$, then
$$
f(r,s_{1},s_{2})=f(r-1,s_{1},s_{2});
$$
if $r-1\equiv s_{1}+2, s_{1}+4,\cdots, s_{1}+s_{2}-1\pmod{2(s_{1}+s_{2})}$, then
$$
f(r,s_{1},s_{2})=f(r-1,s_{1},s_{2})+1;
$$
if $r-1\equiv s_{1}+s_{2}+1, s_{1}+s_{2}+2,\cdots, 2s_{1}+s_{2}\pmod{2(s_{1}+s_{2})}$, then
$$
f(r,s_{1},s_{2})=f(r-1,s_{1},s_{2});
$$
if $r-1\equiv 2s_{1}+s_{2}+1, 2s_{1}+s_{2}+3,\cdots, 2(s_{1}+s_{2})\pmod{2(s_{1}+s_{2})}$, then
$$
f(r,s_{1},s_{2})=f(r-1,s_{1},s_{2})+1;
$$
if $r-1\equiv 2s_{1}+s_{2}+2, 2s_{1}+s_{2}+4,\cdots, 2(s_{1}+s_{2})-1\pmod{2(s_{1}+s_{2})}$, then
$$
f(r,s_{1},s_{2})=f(r-1,s_{1},s_{2}).\\
$$
Case $s_{2}$=even,
if $r-1\equiv 1, 2,\cdots, s_{1}\pmod{s_{1}+s_{2}}$, then
$$
f(r,s_{1},s_{2})=f(r-1,s_{1},s_{2})+1\,;
$$
if $r-1\equiv s_{1}+1, s_{1}+3,\cdots, s_{1}+s_{2}-1\pmod{s_{1}+s_{2}}$, then
$$
f(r,s_{1},s_{2})=f(r-1,s_{1},s_{2})\,;
$$
if $r-1\equiv s_{1}+2, s_{1}+4,\cdots, s_{1}+s_{2}\pmod{s_{1}+s_{2}}$, then
$$
f(r,s_{1},s_{2})=f(r-1,s_{1},s_{2})+1\,.
$$
\end{Definition}

We now present main results of this section. We omit their proofs which are similar to those of Lemmas \ref{lemma1},\ref{lemma12} and Theorems \ref{maintheorem},\ref{maintheorem12}.

\begin{Lem}\label{mainlemma}
For $r, n, t \in \mathbb N$, we have
\begin{align*}
&\quad T(r,s_{1},s_{2},n,t)\\
&=\sum_{m=0}^{f(r,s_{1},s_{2})}\sum_{j=0}^{f(r,s_{1},s_{2})-m}\biggl(b(r,s_{1},s_{2},m,j,0)+b(r,s_{1},s_{2},m,j,1)(-1)^{n-1+t-1}\\
&\qquad+b(r,s_{1},s_{2},m,j,2)(-1)^{n-1}+b(r,s_{1},s_{2},m,j,3)(-1)^{t-1}\biggr)t^{j} n^{m}\,.
\end{align*}

For $k=0,1,2,3$, $b(r,s_{1},s_{2},m,j,k)$ satisfy the following recurrence relations. If $s_{2}$=odd, then we get that:\\
Case $I$, if $r\equiv 1, 2,\cdots, s_{1}\pmod{2(s_{1}+s_{2})}$, then we have $f(r+1,s_{1},s_{2})=f(r,s_{1},s_{2})+1$,
\begin{align*}
&b(r+1,s_{1},s_{2},m,j,0)=\sum_{\ell=m-1}^{f(r,s_{1},s_{2})}b(r,s_{1},s_{2},\ell,j,0) c(\ell,m)\,,\\
&(1 \le m \le f(r,s_{1},s_{2})+1,\quad 0 \le j \le f(r,s_{1},s_{2})+1-m)\,;\\
&\quad b(r+1,s_{1},s_{2},0,j,0)\\
&=-\sum_{m=0}^{f(r,s_{1},s_{2})}\sum_{\substack{j_{1}+\ell=j\\ 0 \le j_{1} \le f(r,s_{1},s_{2})-m\\ 1 \le \ell \le m+1 }}b(r,s_{1},s_{2},m,j_{1},0) d(m,\ell)+b(r,s_{1},s_{2},0,j,0)\\
&\quad +\sum_{m=0}^{f(r,s_{1},s_{2})}\sum_{\substack{j_{1}+\ell=j\\ 0 \le j_{1} \le f(r,s_{1},s_{2})-m\\ 0 \le \ell \le m }}b(r,s_{1},s_{2},m,j_{1},1) d_{1}(m,\ell)\quad(0 \le j \le f(r,s_{1},s_{2})+1)\,;\\
&b(r+1,s_{1},s_{2},m,j,1)=\sum_{\ell=m}^{f(r,s_{1},s_{2})}b(r,s_{1},s_{2},\ell,j,1) c_{1}(\ell,m)\,,\\
&(0 \le m \le f(r,s_{1},s_{2}),\quad 0 \le j \le f(r,s_{1},s_{2})-m)\,.
\end{align*}

Case $II$, if $r\equiv s_{1}+1, s_{1}+3,\cdots, s_{1}+s_{2}\pmod{2(s_{1}+s_{2})}$ or $r\equiv 2s_{1}+s_{2}+2, 2s_{1}+s_{2}+4,\cdots, 2(s_{1}+s_{2})-1\pmod{2(s_{1}+s_{2})}$, then we have $f(r+1,2,1)=f(r,2,1)$,
\begin{align*}
&b(r+1,s_{1},s_{2},m,j,2)=\sum_{\ell=m}^{f(r,s_{1},s_{2})}b(r,s_{1},s_{2},\ell,j,0) c_{1}(\ell,m)\,,\\
&(0 \le m \le f(r,s_{1},s_{2}),\quad 0 \le j \le f(r,s_{1},s_{2})-m)\,;\\
&b(r+1,s_{1},s_{2},m,j,3)=\sum_{\ell=m-1}^{f(r,s_{1},s_{2})}b(r,s_{1},s_{2},\ell,j,1) c(\ell,m)\,,\\
&(1 \le m \le f(r,s_{1},s_{2})+1,\quad 0 \le j \le f(r,s_{1},s_{2})+1-m)\,;\\
&\quad b(r+1,s_{1},s_{2},0,j,3)\\
&=\sum_{m=0}^{f(r,s_{1},s_{2})}\sum_{\substack{j_{1}+\ell=j\\ 0 \le j_{1} \le f(r,s_{1},s_{2})-m\\ 0 \le \ell \le m}}b(r,s_{1},s_{2},m,j_{1},0) d_{1}(m,\ell)+b(r,s_{1},s_{2},0,j,1)\\
&\quad -\sum_{m=0}^{f(r,s_{1},s_{2})}\sum_{\substack{j_{1}+\ell=j\\ 0 \le j_{1} \le f(r,s_{1},s_{2})-m\\ 1 \le \ell \le m+1 }}b(r,s_{1},s_{2},m,j_{1},1) d(m,\ell)\quad (0 \le j \le f(r,s_{1},s_{2})+1)\,.
\end{align*}

Case $III$, if $r\equiv s_{1}+s_{2}+1, s_{1}+s_{2}+2,\cdots, 2s_{1}+s_{2}\pmod{2(s_{1}+s_{2})}$, then we have $f(r+1,s_{1},s_{2})=f(r,s_{1},s_{2})$,
\begin{align*}
&b(r+1,s_{1},s_{2},m,j,2)=\sum_{\ell=m}^{f(r,s_{1},s_{2})}b(r,s_{1},s_{2},\ell,j,2) c_{1}(\ell,m)\,,\\
&(0 \le m \le f(r,s_{1},s_{2}),\quad 0 \le j \le f(r,s_{1},s_{2})-m)\,;\\
&b(r+1,s_{1},s_{2},m,j,3)=\sum_{\ell=m-1}^{f(r,s_{1},s_{2})}b(r,s_{1},s_{2},\ell,j,3) c(\ell,m)\,,\\
&(1 \le m \le f(r,s_{1},s_{2})+1,\quad 0 \le j \le f(r,s_{1},s_{2})+1-m)\,;\\
&\quad b(r+1,s_{1},s_{2},0,j,3)\\
&=\sum_{m=0}^{f(r,s_{1},s_{2})}\sum_{\substack{j_{1}+\ell=j\\ 0 \le j_{1} \le f(r,s_{1},s_{2})-m\\ 0 \le \ell \le m}}b(r,s_{1},s_{2},m,j_{1},2) d_{1}(m,\ell)+b(r,s_{1},s_{2},0,j,3)\\
&\quad -\sum_{m=0}^{f(r,s_{1},s_{2})}\sum_{\substack{j_{1}+\ell=j\\ 0 \le j_{1} \le f(r,s_{1},s_{2})-m\\ 1 \le \ell \le m+1 }}b(r,s_{1},s_{2},m,j_{1},3) d(m,\ell)\quad (0 \le j \le f(r,s_{1},s_{2})+1)\,.
\end{align*}

Case $IV$, if $r\equiv s_{1}+2, s_{1}+4,\cdots, s_{1}+s_{2}-1\pmod{2(s_{1}+s_{2})}$ or $r\equiv 2s_{1}+s_{2}+1, 2s_{1}+s_{2}+3,\cdots, 2(s_{1}+s_{2})\pmod{2(s_{1}+s_{2})}$, then we have $f(r+1,s_{1},s_{2})=f(r,s_{1},s_{2})+1$,
\begin{align*}
&b(r+1,s_{1},s_{2},m,j,0)=\sum_{\ell=m-1}^{f(r,s_{1},s_{2})}b(r,s_{1},s_{2},\ell,j,2) c(\ell,m)\,,\\
&(1 \le m \le f(r,s_{1},s_{2})+1,\quad 0 \le j \le f(r,s_{1},s_{2})+1-m)\,;\\
&\quad b(r+1,s_{1},s_{2},0,j,0)\\
&=-\sum_{m=0}^{f(r,s_{1},s_{2})}\sum_{\substack{j_{1}+\ell=j\\ 0 \le j_{1} \le f(r,s_{1},s_{2})-m\\ 1 \le \ell \le m+1 }}b(r,s_{1},s_{2},m,j_{1},2) d(m,\ell)+b(r,s_{1},s_{2},0,j,2)\\
&\quad +\sum_{m=0}^{f(r,s_{1},s_{2})}\sum_{\substack{j_{1}+\ell=j\\ 0 \le j_{1} \le f(r,s_{1},s_{2})-m\\ 0 \le \ell \le m }}b(r,s_{1},s_{2},m,j_{1},3) d_{1}(m,\ell)\quad (0 \le j \le f(r,s_{1},s_{2})+1)\,;\\
&b(r+1,s_{1},s_{2},m,j,1)=\sum_{\ell=m}^{f(r,s_{1},s_{2})}b(r,s_{1},s_{2},\ell,j,3) c_{1}(\ell,m)\,,\\
&(0 \le m \le f(r,s_{1},s_{2}),\quad 0 \le j \le f(r,s_{1},s_{2})-m)\,.
\end{align*}

If $s_{2}$=even, then we get that:\\
Case $I^{\prime}$, if $r\equiv 1, 2,\cdots, s_{1}\pmod{s_{1}+s_{2}}$, then we have $f(r+1,s_{1},s_{2})=f(r,s_{1},s_{2})+1$,
\begin{align*}
&b(r+1,s_{1},s_{2},m,j,0)=\sum_{\ell=m-1}^{f(r,s_{1},s_{2})}b(r,s_{1},s_{2},\ell,j,0) c(\ell,m)\,,\\
&(1 \le m \le f(r,s_{1},s_{2})+1,\quad 0 \le j \le f(r,s_{1},s_{2})+1-m)\,;\\
&\quad b(r+1,s_{1},s_{2},0,j,0)\\
&=-\sum_{m=0}^{f(r,s_{1},s_{2})}\sum_{\substack{j_{1}+\ell=j\\ 0 \le j_{1} \le f(r,s_{1},s_{2})-m\\ 1 \le \ell \le m+1 }}b(r,s_{1},s_{2},m,j_{1},0) d(m,\ell)+b(r,s_{1},s_{2},0,j,0)\\
&\quad +\sum_{m=0}^{f(r,s_{1},s_{2})}\sum_{\substack{j_{1}+\ell=j\\ 0 \le j_{1} \le f(r,s_{1},s_{2})-m\\ 0 \le \ell \le m }}b(r,s_{1},s_{2},m,j_{1},1) d_{1}(m,\ell)\quad (0 \le j \le f(r,s_{1},s_{2})+1)\,;\\
&b(r+1,s_{1},s_{2},m,j,1)=\sum_{\ell=m}^{f(r,s_{1},s_{2})}b(r,s_{1},s_{2},\ell,j,1) c_{1}(\ell,m)\,,\\
&(0 \le m \le f(r,s_{1},s_{2}),\quad 0 \le j \le f(r,s_{1},s_{2})-m)\,.
\end{align*}

Case $II^{\prime}$, if $r\equiv s_{1}+1, s_{1}+3,\cdots, s_{1}+s_{2}-1\pmod{s_{1}+s_{2}}$, then we have $f(r+1,s_{1},s_{2})=f(r,s_{1},s_{2})$,
\begin{align*}
&b(r+1,s_{1},s_{2},m,j,2)=\sum_{\ell=m}^{f(r,s_{1},s_{2})}b(r,s_{1},s_{2},\ell,j,0) c_{1}(\ell,m)\,,\\
&(0 \le m \le f(r,s_{1},s_{2}),\quad 0 \le j \le f(r,s_{1},s_{2})-m)\,;\\
&b(r+1,s_{1},s_{2},m,j,3)=\sum_{\ell=m-1}^{f(r,s_{1},s_{2})}b(r,s_{1},s_{2},\ell,j,1) c(\ell,m)\,,\\
&(1 \le m \le f(r,s_{1},s_{2})+1,\quad 0 \le j \le f(r,s_{1},s_{2})+1-m)\,;\\
&\quad b(r+1,s_{1},s_{2},0,j,3)\\
&=\sum_{m=0}^{f(r,s_{1},s_{2})}\sum_{\substack{j_{1}+\ell=j\\ 0 \le j_{1} \le f(r,s_{1},s_{2})-m\\ 0 \le \ell \le m}}b(r,s_{1},s_{2},m,j_{1},0) d_{1}(m,\ell)+b(r,s_{1},s_{2},0,j,1)\\
&\quad -\sum_{m=0}^{f(r,s_{1},s_{2})}\sum_{\substack{j_{1}+\ell=j\\ 0 \le j_{1} \le f(r,s_{1},s_{2})-m\\ 1 \le \ell \le m+1 }}b(r,s_{1},s_{2},m,j_{1},1) d(m,\ell)\quad (0 \le j \le f(r,s_{1},s_{2})+1)\,.
\end{align*}

Case $III^{\prime}$, if $r\equiv s_{1}+2, s_{1}+4,\cdots, s_{1}+s_{2}\pmod{s_{1}+s_{2}}$, then we have $f(r+1,s_{1},s_{2})=f(r,s_{1},s_{2})+1$,
\begin{align*}
&b(r+1,s_{1},s_{2},m,j,0)=\sum_{\ell=m-1}^{f(r,s_{1},s_{2})}b(r,s_{1},s_{2},\ell,j,2) c(\ell,m)\,,\\
&(1 \le m \le f(r,s_{1},s_{2})+1,\quad 0 \le j \le f(r,s_{1},s_{2})+1-m)\,;\\
&\quad b(r+1,s_{1},s_{2},0,j,0)\\
&=-\sum_{m=0}^{f(r,s_{1},s_{2})}\sum_{\substack{j_{1}+\ell=j\\ 0 \le j_{1} \le f(r,s_{1},s_{2})-m\\ 1 \le \ell \le m+1 }}b(r,s_{1},s_{2},m,j_{1},2) d(m,\ell)+b(r,s_{1},s_{2},0,j,2)\\
&\quad +\sum_{m=0}^{f(r,s_{1},s_{2})}\sum_{\substack{j_{1}+\ell=j\\ 0 \le j_{1} \le f(r,s_{1},s_{2})-m\\ 0 \le \ell \le m }}b(r,s_{1},s_{2},m,j_{1},3) d_{1}(m,\ell)\quad (0 \le j \le f(r,s_{1},s_{2})+1)\,;\\
&b(r+1,s_{1},s_{2},m,j,1)=\sum_{\ell=m}^{f(r,s_{1},s_{2})}b(r,s_{1},s_{2},\ell,j,3) c_{1}(\ell,m)\,,\\
&(0 \le m \le f(r,s_{1},s_{2}),\quad 0 \le j \le f(r,s_{1},s_{2})-m)\,.
\end{align*}
The initial value is given by $T(1,s_{1},s_{2},n,t)= 1$.
\end{Lem}

We have the following theorem.
\begin{theorem}%\label{maintheorem12}
Let $r,p \in \mathbb N$ and $q$ be a positive real number with $q \geq r+1$, we have,
\begin{align*}
&\quad \sum_{n=1}^\infty \frac{H_n^{(p,r,s_{1},s_{2})}}{n^{q}}\\
&=\sum_{m=0}^{f(r,s_{1},s_{2})}\sum_{j=0}^{f(r,s_{1},s_{2})-m}
\biggl(b(r,s_{1},s_{2},j,m,0)S_{p-m,q-j}^{+,+}+b(r,s_{1},s_{2},j,m,1)S_{p-m,q-j}^{-,-}\\
&\qquad+b(r,s_{1},s_{2},j,m,2)S_{p-m,q-j}^{+,-}+b(r,s_{1},s_{2},j,m,3)S_{p-m,q-j}^{-,+}\biggr).
\end{align*}
Therefore Euler sums of the generalized alternating hyperharmonic numbers $H_n^{(p,r,s_{1},s_{2})}$ can be expressed in terms of linear combinations of classical (alternating) Euler sums.
\end{theorem}

We now present the following three conjectures on the coefficients.
\begin{Conjecture}
For $r, m, j \in \mathbb N$ with $0 \leq m \leq f(r,s_{1},s_{2})$ and $0 \leq j \leq f(r,s_{1},s_{2})-m$, we conjecture that:
$$
b(r,s_{1},s_{2},m,j,k)=(-1)^{m+j} b(r,s_{1},s_{2},j,m,k) \quad (k=0,1,2,3).
$$
\end{Conjecture}

\begin{Conjecture}
For $r, m \in \mathbb N$, we conjecture that:
$$
\sum_{j=0}^{m}\sum_{k=0}^{3}b(r,s_{1},s_{2},m-j,j,k)=\delta_{m 0} \quad \bigl(0 \leq m \leq f(r,s_{1},s_{2})\bigr),
$$
where $\delta_{m n}$ is the Kronecker delta, that is, $\delta_{m m}=1$, $\delta_{m n}=0$ for $m \neq n$.
\end{Conjecture}

\begin{Conjecture}
For $r, m \in \mathbb N$ with $0\leq m \leq f(r,s_{1},s_{2})$, we conjecture that:
for $0 \leq m \leq f(r,s_{1},s_{2})$ and $0 \leq j \leq f(r,s_{1},s_{2})-m$, we have
\begin{align*}
\mathrm{sgn}(\sum_{k=0}^{3}b(r,s_{1},s_{2},m,j,k))=(-1)^{j}\,,
\end{align*}
where $\mathrm{sgn}(x)$ is the signum function defined by
\begin{equation*}
\mathrm{sgn}(x)=
\begin{cases}
1& \text{$x>0$},\\
0& \text{$x=0$},\\
-1& \text{$x<0$}.
\end{cases}
\end{equation*}
In particular for $0 \leq m \leq f(r,s_{1},s_{2})$ and $0 \leq j \leq f(r,s_{1},s_{2})-m$, we have
\begin{align*}
\sum_{k=0}^{3} b(r,s_{1},s_{2},m,j,k)\neq 0\,.
\end{align*}
\end{Conjecture}

\section{Nonlinear Euler sums of generalized alternating hyperharmonic numbers $H_n^{(p,r,s_{1},s_{2})}$}

In this section, we investigate nonlinear Euler sums of the generalized alternating hyperharmonic numbers $H_n^{(p,r,s_{1},s_{2})}$.

\begin{theorem}
Let $k, r_{1}, \cdots, r_{k},p_{1}, p_{k} \in \mathbb N$, $s_{1,1},\cdots,s_{k,1}, s_{1,2},\cdots s_{k,2}\in \mathbb N \cup \{0\}, 1\leq i \leq k, s_{i,1}+s_{i,2} \geq 1$ and $q$ be a positive real number large enough, we can obtain that
nonlinear Euler sums of the generalized alternating hyperharmonic numbers $H_n^{(p,r,s_{1},s_{2})}$ can be expressed in terms of linear combinations of classical nonlinear (alternating) Euler sums.
\end{theorem}
\begin{proof}
Using Lemma \ref{lemma1}, we can write
\begin{align*}
&\quad \sum_{n=1}^\infty \frac{H_n^{(p_{1},r_{1},s_{1,1},s_{1,2})}\cdots H_n^{(p_{k},r_{k},s_{k,1},s_{k,2})}}{n^{q}}\\
&=\sum_{n=1}^\infty \frac{1}{n^{q}}\prod_{i=1}^{k}\sum_{m_{i}=0}^{f(r_{i},s_{i,1},s_{i,2})}
\sum_{j_{i}=0}^{f(r_{i},s_{i,1},s_{i,2})-m_{i}}\biggl(b(r_{i},s_{i,1},s_{i,2},j_{i},m_{i},0)n^{j_{i}} H_{n}^{(p_{i}-m_{i})}\\
&\quad+b(r_{i},s_{i,1},s_{i,2},j_{i},m_{i},1)(-1)^{n-1}n^{j_{i}} \bar{H}_{n}^{(p_{i}-m_{i})}+b(r_{i},s_{i,1},s_{i,2},j_{i},m_{i},2)\\
&\quad\times(-1)^{n-1}n^{j_{i}} H_{n}^{(p_{i}-m_{i})}+b(r_{i},s_{i,1},s_{i,2},j_{i},m_{i},3)n^{j_{i}}\bar{H}_{n}^{(p_{i}-m_{i})}\biggr) \,\\
&=\sum_{m_{1}=0}^{f(r,s_{1},s_{2})}\sum_{j_{1}=0}^{f(r,s_{1},s_{2})-m_{1}}\cdots \sum_{m_{k}=0}^{f(r,s_{1},s_{2})}\sum_{j_{k}=0}^{f(r,s_{1},s_{2})-m_{k}}\sum_{n=1}^\infty \frac{1}{n^{q}}\\
&\quad \times \prod_{i=1}^{k}\biggl(b(r_{i},s_{i,1},s_{i,2},j_{i},m_{i},0)n^{j_{i}} H_{n}^{(p_{i}-m_{i})}+b(r_{i},s_{i,1},s_{i,2},j_{i},m_{i},1)(-1)^{n-1}\\
&\quad \times n^{j_{i}}\bar{H}_{n}^{(p_{i}-m_{i})}+b(r_{i},s_{i,1},s_{i,2},j_{i},m_{i},2)(-1)^{n-1}n^{j_{i}} H_{n}^{(p_{i}-m_{i})}\\
&\quad +b(r_{i},s_{i,1},s_{i,2},j_{i},m_{i},3)n^{j_{i}}\bar{H}_{n}^{(p_{i}-m_{i})}\biggr).
\end{align*}
Since
\begin{align*}
&\sum_{n=1}^\infty \frac{1}{n^{q}}\prod_{i=1}^{k}\biggl(b(r_{i},s_{i,1},s_{i,2},j_{i},m_{i},0)n^{j_{i}} H_{n}^{(p_{i}-m_{i})}+b(r_{i},s_{i,1},s_{i,2},j_{i},m_{i},1)(-1)^{n-1}\\
&\quad \times n^{j_{i}}\bar{H}_{n}^{(p_{i}-m_{i})}+b(r_{i},s_{i,1},s_{i,2},j_{i},m_{i},2)(-1)^{n-1}n^{j_{i}} H_{n}^{(p_{i}-m_{i})}\\
&\quad +b(r_{i},s_{i,1},s_{i,2},j_{i},m_{i},3)n^{j_{i}}\bar{H}_{n}^{(p_{i}-m_{i})}\biggr)
\end{align*}
is reduced to linear combinations of classical nonlinear (alternating) Euler sums, we obtain the desired result.
\end{proof}

Recently, the author \cite{Lirusensen,LirusensenJMAA} considered generalized (alternating) hyperharmonic number sums with reciprocal binomial coefficients. Combining Lemma \ref{mainlemma} and lemmata proved in the author's previous works \cite{Lirusensen,LirusensenJMAA}, we have the following result.
\begin{Remark}
Generalized alternating hyperharmonic number sums
\begin{align*}
&\sum_{n=1}^\infty\frac{H_n^{(p,r,s_{1},s_{2})}}{n^{m}\binom{n+k}{k}}\,,\quad
\sum_{n=1}^\infty\frac{(-1)^{n+1}H_n^{(p,r,s_{1},s_{2})}}{n^{m}\binom{n+k}{k}}\,,\quad
\sum_{n=1}^\infty\frac{H_n^{(p_{1},r_{1},s_{1},t_{1})}H_n^{(p_{2},r_{2},s_{2},t_{2})}}{n^{m}\binom{n+k}{k}}\,,\quad\\
&\sum_{n=1}^\infty\frac{(-1)^{n+1}H_n^{(p_{1},r_{1},s_{1},t_{1})}H_n^{(p_{2},r_{2},s_{2},t_{2})}}{n^{m}\binom{n+k}{k}}
\end{align*}
can be expressed in terms of classical (alternating) Euler sums, zeta values and generalized (alternating) harmonic numbers.
\end{Remark}

\section{Some open problems}

We now provide some open problems for further considerations.
\begin{Problem}
It is well known (see {\cite{Benjamin}}) that the (hyper-)harmonic numbers are closely related to ($r$-)Stirling numbers:
\begin{align*}
H_{n}=\frac{\stf{n+1}{2}}{n!}\,,
h_{n}^{r}=\frac{\stf{n+r}{r+1}_r}{n!}\,,
\end{align*}
where where $\stf{n}{k}$ and $\stf{n}{k}_{r}$ denote the (unsigned) Stirling number of the first kind and $r$-Stirling number, respectively. So we may ask such a question: Can we find the corresponding Stirling numbers for the generalized alternating hyperharmonic numbers $H_n^{(p,r,s_{1},s_{2})}$?
\end{Problem}

\begin{Problem}
G\"oral and Sertba\c s (see {\cite{Goral}}) studied divisibility properties of hyperharmonic numbers and extend Wolstenholme's theorem to them. Moreover, they proved that all hyperharmonic numbers in their reduced fractional form are odd and provided $p$-adic value lower bounds for certain hyperharmonic numbers. Can we extend Wolstenholme's theorem to the generalized alternating hyperharmonic numbers $H_n^{(p,r,s_{1},s_{2})}$? Is it possible to obtain similar divisibility properties for the generalized alternating hyperharmonic numbers $H_n^{(p,r,s_{1},s_{2})}$?
\end{Problem}

\begin{Problem}
Kamano \cite{Kamano} proved the analytic continuation of  $\sum_{n=1}^\infty h_n^{(r)}/{n^{s}}$. Is it possible to obtain similar analytic continuation for Euler sums of the generalized alternating hyperharmonic numbers $H_n^{(p,r,s_{1},s_{2})}$? Similar with the infamous Riemann Hypothesis, can we study the nontrivial zeros of this complex function $\sum_{n=1}^\infty H_n^{(p,r,s_{1},s_{2})}/{n^{s}}$?
\end{Problem}

\begin{Problem}
\" Om\" ur and Koparal \cite{omur} defined two $n\times n$ matrices $A_n$ and $B_n$ with $a_{i,j}=H_i^{(j,r)}$ and $b_{i,j}=H_i^{(p,j)}$, respectively, and gave some interesting factorizations and determinant properties of the matrices $A_n$ and $B_n$. Is it possible to obtain similar determinant properties for the generalized alternating hyperharmonic numbers $H_n^{(p,r,s_{1},s_{2})}$?
\end{Problem}

\section{Data availability}

The datasets generated during and/or analysed during the current study are available from the corresponding author on reasonable request.

\end{document}